\documentclass[
a4paper,
reqno,
11pt,
]{article}
\usepackage{
amssymb,
amsmath,
amsthm,
eucal,
dsfont,
}
\makeatletter
 
  \@addtoreset{equation}{section}
 \makeatother
\theoremstyle{plain}

\newtheorem{theorem}{Theorem}[section]
\newtheorem{lemma}[theorem]{Lemma}
\newtheorem{proposition}[theorem]{Proposition}
\newtheorem{corollary}[theorem]{Corollary}
\theoremstyle{definition}
\newtheorem{definition}[theorem]{Definition}
\newtheorem{remark}[theorem]{Remark}

\newtheorem{assumption}[theorem]{Assumption}

\newcommand{\bignorm}[1]{{\Big|\Big|#1\Big|\Big|}}
\newcommand{\norm}[1]{{||#1||}}

\newcommand{\wtilde}[1]{{\widetilde{#1}}}

\def\Id{\mathop{\mathrm{Id}}\nolimits}

\def\Ran{\mathop{\mathrm{Ran}}\nolimits}

\def\Re{\mathop{\mathrm{Re}}\nolimits}
\def\Im{\mathop{\mathrm{Im}}\nolimits}
\def\loc{\mathop{\mathrm{loc}}\nolimits}

\def\Hardy{\mathop{\mathrm{H}}\nolimits}

\def\R{{\mathbb{R}}}

\def\C{{\mathbb{C}}}

\def\S{{\mathcal{S}}}

\def\H{{\mathcal{H}}}
\def\A{{\mathcal{A}}}
\def\B{{\mathcal{B}}}

\def\<{{\langle}}
\def\>{{\rangle}}

\def\ep{{\varepsilon}}

\title
{Remarks on endpoint Strichartz estimates for Schr\"odinger equations with the critical inverse-square potential}
\author{Haruya Mizutani\footnote{Department of Mathematics, Graduate School of Science, Osaka University, Toyonaka, Osaka 560-0043, Japan. E-mail address: \texttt{haruya@math.sci.osaka-u.ac.jp}}
}
\date{\empty}

\begin{document}
\maketitle

\begin{abstract}
The purpose of this paper is to study the validity of global-in-time Strichartz estimates for the Schr\"odinger equation on $\R^n$, $n\ge3$, with the negative inverse-square potential $-\sigma|x|^{-2}$ in the critical case $\sigma=(n-2)^2/4$. It turns out that the situation is different from the subcritical case $\sigma<(n-2)^2/4$ in which the full range of Strichartz estimates is known to be hold. More precisely, splitting the solution into the radial and non-radial parts, we show that (i) the radial part satisfies a weak-type endpoint estimate, which can be regarded as an extension to higher dimensions of the endpoint Strichartz estimate with radial data for the two-dimensional free Schr\"odinger equation; (ii) other endpoint estimates in Lorentz spaces for the radial part fail in general; (iii) the non-radial part satisfies the full range of Strichartz estimates. 
\end{abstract}



\section{Introduction}
\label{section_Introduction}
This paper is concerned with global-in-time dispersive properties of the unitary group $e^{-itH}$ for the Schr\"odinger operator with the inverse-square potential of the form
$$
H:=-\Delta-C_{\Hardy}|x|^{-2},\quad x\in \R^n,
$$
where $n\ge3$, $\Delta=\sum_{j=1}^n\partial_{x_j}^2$ is the Laplacian  and
$$
C_{\Hardy}:=\frac{(n-2)^2}{4}
$$
is the best constant in Hardy's inequality: 
\begin{align}
\label{Hardy}
C_{\Hardy}\int  |x|^{-2}|u|^2dx\le \int|\nabla u|^2dx,\quad u\in C_0^\infty(\R^n).
\end{align}
In particular, we are interested in the validity of Strichartz estimates for $e^{-itH}$. 

Let us first recall the free case. It is well known (see \cite{Str,GiVe,KeTa}) that the free Schr\"odinger evolution group $e^{it\Delta}$ satisfies the following family of space-time inequalities, known as {\it homogeneous and inhomogeneous Strichartz estimates}, respectively:
\begin{align}
\label{Strichartz_1}
\norm{e^{it\Delta}\psi}_{L^p(\R;L^q(\R^n))}&\le C\norm{\psi}_{L^2(\R^n)},\\
\label{Strichartz_2}\bignorm{\int_0^te^{i(t-s)\Delta}F(s)ds}_{L^p(\R;L^q(\R^n))}&\le C\norm{F}_{L^{\tilde p'}(\R;L^{\tilde q'}\R^n))}
\end{align}
for admissible pairs $(p,q)$ and $(\tilde p,\tilde q)$, where $p'=p/(p-1)$ denotes the H\"older conjugate of $p$,  and $(p,q)$ is said to be an ($n$-dimensional) admissible pair if 
\begin{align}
\label{admissible}
p,q\ge2,\quad 2/p+n/q=n/2,\quad(p,q,n)\neq(2,\infty,2).
\end{align}
The condition \eqref{admissible} is  necessary and sufficient for the validity of \eqref{Strichartz_1}. In particular, as shown by \cite{Mon,Tao1}, the two-dimensional endpoint estimates  do not hold in general. The proof of \eqref{Strichartz_1} and \eqref{Strichartz_2} is based on the dispersive estimate of the form 
\begin{align}
\label{dispersive}
\norm{e^{it\Delta}\psi}_{L^\infty(\R^n)}\le C|t|^{-n/2}\norm{\psi}_{L^1(\R^n)},\quad t\neq0. 
\end{align}
In fact, \eqref{dispersive} together with the mass conservation $\norm{e^{it\Delta}\psi}_{L^2}=\norm{\psi}_{L^2}$ implies both \eqref{Strichartz_1} and \eqref{Strichartz_2} for all admissible pairs (see \cite{KeTa}). It is worth noting that the Strichartz estimates play an important role in the  study of the well-posedness and scattering theory for nonlinear Schr\"odinger equations (see, {\it e.g.}, monographs \cite{Caz,Tao2}). 

There is also a vast literature on Strichartz estimates for Schr\"odinger equations with potentials  (see \cite{RoSc,BPST1,BPST2,EGS1,EGS2,Gol,DaFa,MMT,DFVV,Bec,KiKo,BoMi} and references therein). In particular, the Schr\"odinger operator with inverse-square potentials of the form $H_\sigma=-\Delta-\sigma|x|^{-2}$ with $\sigma\in \R$ has recently been extensively studied since it represents a borderline case as follows. Both the decay rate $|x|^{-2}$ as $|x|\to+\infty$ and the local singularity $|x|^{-2}$ as $|x|\to0$ are critical for the validity of Strichartz and dispersive estimates (see \cite{GVV,Duy}). 
Moreover, the case with $\sigma=C_{\Hardy}$, {\it i.e.}, $H_\sigma=H$, is also critical in the following sense. On one hand, if $\sigma>C_{\Hardy}$, $H_\sigma$ is not lower semi-bounded (due to the fact that $C_{\Hardy}$ is the best constant in \eqref{Hardy}) and has infinitely many negative eigenvalues diverging to $-\infty$ (see \cite{PST1}). In particular, there is no hope to obtain any kind of global-in-time dispersive estimates. On the other hand, if $\sigma<C_{\Hardy}$, the full range of Strichartz estimates is known to be hold (see \cite{BPST1,BPST2,BoMi}) and has been used to study the nonlinear scattering theory (see \cite{ZhZh,KMVZZ,KMVZ}). Note that the method of the proof in the subcritical case essentially relies on the fact that $H_\sigma$ satisfies $-C_1\Delta\le H_\sigma\le -C_2\Delta$ with some constants $C_1,C_2>0$ in the sense of forms, which is not the case for the critical case. In this subcritical case, there are also several results on dispersive estimates for Schr\"odinger equations (see \cite{FFFP1,FFFP2,KoTr}) or wave equations (see \cite{PST1,PST2}). 

In the critical case $\sigma=C_{\Hardy}$, Burq--Planchon--Stalker--Tahvildar-Zadeh \cite{BPST1} obtained local smoothing effects of the form 
$$
\norm{|x|^{-1/2-2\ep}H^{1/4-\ep}e^{-itH}\psi}_{L^2(\R^{1+n})}\le C_\ep\norm{\psi}_{L^2(\R^n)},\quad\ep>0. 
$$
We refer to \cite{Fan,Moc,BVZ,Dan,CaSe,CaFa} and references therein for further references on such weak-dispersive estimates for dispersive equations with singular perturbations. More recently, Suzuki \cite{Suz} proved Strichartz estimates for all admissible pairs {\it except for the endpoint} $(p,q)=(2,2n/(n-2))$ and used them to study the well-posedness of nonlinear Schr\"odinger equations. We also mention that, in the one dimensional case $n=1$, \cite{KoTr} considered $H_\sigma$ in $L^2(\R_+)$ with Dirichlet boundary condition at $x=0$ and showed the dispersive estimate $\norm{e^{-itH_\sigma}}_{L^1\to L^\infty}\le |t|^{-1/2}$ for $\sigma\ge -1/4$. However, there seems to be no previous literature on the endpoint Strichartz estimate if $n\ge3$. 

The main purpose of this paper is to provide a complete answer, in the framework of Lorentz spaces, to the validity of endpoint Strichartz estimates in the critical case $\sigma=C_{\Hardy}$, which turns out to be different from the subcritical case. More precisely, splitting the solution into the radial and non-radial parts, we show that
\begin{itemize}
\item the radial part satisfies a weak-type endpoint estimate, but other endpoint estimates in Lorentz spaces do not hold in general. In particular, the usual endpoint estimates can fail;
\item the non-radial part obeys the full range of Strichartz estimates.
\end{itemize}
We also consider a more general Schr\"odinger operator of the form 
$$
H_V=-\Delta+V,\quad V(x)=-C_{\Hardy}|x|^{-2}+\wtilde V(|x|),
$$
where $\widetilde V$ is a $C^1$ function on $\R_+$ such that $|\wtilde V(r)|\le C\<r\>^{-\mu}$ with some $\mu>3$ and the negative part of $\wtilde V$ is small enough. Then, under appropriate spectral conditions at zero energy, we show that $e^{-itH_V}(\Id-\mathbb P_{\mathrm {pp}})$ satisfies non-endpoint Strichartz estimates, where we refer to Section \ref{section_generalization} for the definition of $\mathbb P_{\mathrm {pp}}$.


\section{Main results}
\label{results}
In what follows we always assume $n\ge3$. Let us define the operator $
H =-\Delta-C_{\Hardy}|x|^{-2}
$
 as the Friedrichs extension of the sesquilinear form
$$
 Q_H(u,v)=\int \Big(\nabla u\cdot\overline{\nabla v}-C_{\Hardy}|x|^{-2}u\overline v\Big)dx,\quad u,v\in C_0^\infty(\R^n),
$$
which is symmetric, non-negative and closable by \eqref{Hardy}. To be more precise, let $ \overline Q_H$ be the closure of $Q_H$ with domain $ D(\overline Q_H)$ given by the completion of $C_0^\infty(\R^n)$ with respect to the norm $(\norm{u}_{L^2}^2+Q_H(u,u))^{1/2}$. Then we define a unique self-adjoint operator $H$ corresponding to $\overline Q_H$ in a usual way (see \cite[Chapter VIII]{ReSi}) that $$D(H)=\{u\in D( \overline Q_H)\ |\ \text{
 $| \overline Q_H(u,v)|\le C_u\norm{v}_{L^2}$ for all $v\in D( \overline Q_H)$}\}$$and, for each $u\in D(H)$, $Hu$ is the unique element in $L^2(\R^n)$ such that $ \overline Q_H(u,v)=(Hu,v)$ for all $v\in D( \overline Q_H)$. 

In order to state the results, we further introduce several notation. Let $L^{p,q}(\R^n)$ be the Lorentz space (see the end of this section) and we say that $f\in L^{p,q}_{\mathrm{rad}}(\R^n)$ if $f\in L^{p,q}(\R^n)$ and $f$ is radially symmetric. Let $P_n:L^2(\R^n) \to L^2_{\mathrm{rad}}(\R^n)$ be the projection defined by
\begin{align}
\label{projection_1}
P_nf(x)=\frac{1}{\omega_n}\int_{\mathbb S^{n-1}}f(|x|\theta)d\sigma(\theta)
\end{align}
and set $P_n^\perp=\Id-P_n$, where $\omega_n:=|\mathbb S^{n-1}|$ is the area of the unit sphere $\mathbb S^{n-1}$. We often omit the subscript $n$, writing simply $P=P_n$ and $P^\perp=P^\perp_n$ if there is no confusion. Both $P$ and $P^\perp$ commute with $H$ since the potential $-C_{\Hardy}|x|^{-2}$ is radially symmetric. Furthermore, they also commute with $f(H)$ for any bounded Borel function $f$ on $\R$ by the spectral theorem. Then our first result is as follows.
\begin{theorem}
\label{theorem_1}
There exists $C>0$ such that 
\begin{align}
\label{theorem_1_1}
\norm{e^{-itH }P\psi}_{L^2(\R;L^{2^* ,\infty}(\R^n))}&\le C\norm{\psi}_{L^2(\R^n)},\\
\label{theorem_1_2}
\norm{e^{-itH }P^\perp\psi}_{L^2(\R;L^{2^* ,2}(\R^n))}&\le C\norm{\psi}_{L^2(\R^n)},\\
\label{theorem_1_3}
\bignorm{\int_0^t e^{-i(t-s)H}P^\perp F(s)ds}_{L^2(\R;L^{2^*,2}(\R^n))}&\le C\norm{F}_{L^2(\R;L^{2_*,2}(\R^n))}
\end{align}
for all $\psi\in L^2(\R^n)$ and $F\in L^2(\R;L^{2_*,2}(\R^n))$.  In particular, one has
$$
\norm{e^{-itH }\psi}_{L^2(\R;L^{2^* ,\infty}(\R^n))}\le C\norm{\psi}_{L^2(\R^n)},\quad\psi\in L^2(\R^n),
$$
 where $
2^*=\frac{2n}{n-2}$ and $
2_*=\frac{2n}{n+2}$.

\end{theorem}

Precisely stated, \eqref{theorem_1_3} means that \eqref{theorem_1_3} holds for all $F\in L^2(\R;L^{2_* ,2}(\R^n))\cap L^1_{\mathrm{loc}}(\R;L^2(\R^n))$ and the map $F\mapsto \int_0^te^{-i(t-s)H}P^\perp F(s)ds$ extends to a bounded operator from $L^2(\R;L^{2_*,2}(\R^n))$ to $L^2(\R;L^{2^*,2}(\R^n))$. A similar remark also holds for Corollary \ref{corollary_2} below.

\begin{remark}
It was shown by Tao \cite{Tao1} that the two-dimensional free endpoint estimate
\begin{align}
\label{Tao}
\norm{e^{it\Delta}\psi}_{L^2(\R;L^\infty(\R^2))}\le C\norm{\psi}_{L^2(\R^2)}
\end{align}
holds for all $\psi\in L^2_{\mathrm{rad}}(\R^2)$. The estimate \eqref{theorem_1_1} can be regarded as an extension of this estimate to higher dimensions. Indeed, $H$ restricted to radial functions is unitarily equivalent to $-\Delta_{\R^2}$ restricted to radial functions.  \eqref{theorem_1_1} actually follows from \eqref{Tao} and this equivalence. 

On the other hand, the proof for the non-radial part is quite different from that for the radial part.  We first use a smooth perturbation method developed in \cite{RoSc,BPST2,BoMi} to deduce \eqref{theorem_1_2} and \eqref{theorem_1_3} from the uniform estimate for the weighted resolvent $|x|^{-1}P^\perp (H-z)^{-1}P^\perp |x|^{-1}$. Then a key observation is the following improved Hardy inequality
$$
\frac{n^2}{4}\int |x|^{-2}|P^\perp u|^2dx\le \int |\nabla u|^2dx,\quad u\in C_0^\infty(\R^n),
$$
which implies the equivalence $\norm{H^{1/2}P^\perp u}_{L^2}\sim \norm{\nabla P^\perp u}_{L^2}$. By virtue of this equivalence, we can use a similar technique as in the subcritical case to obtain the desired resolvent estimate. This method does not rely on the explicit formula of $H$. Indeed, \eqref{theorem_1_2} and \eqref{theorem_1_3} hold for more general class of radially symmetric potentials with critical singularities (see Proposition  \ref{proposition_4}). 
\end{remark}


As a consequence of Theorem \ref{theorem_1}, we obtain usual non-endpoint estimates and inhomogeneous weak-type endpoint estimates except for the double endpoint case.
\begin{corollary}
\label{corollary_2}
Let $(p,q)$ and $(\tilde p,\tilde q) $ satisfy \eqref{admissible} and $p,\tilde p>2$. Then one has
\begin{align}
\label{corollary_2_1}
\norm{e^{-itH }\psi}_{L^p(\R;L^q(\R^n))}&\le C\norm{\psi}_{L^2(\R^n)},\\
\label{corollary_2_2}
\bignorm{\int_0^te^{-i(t-s)H}F(s)ds}_{L^p(\R;L^{q}(\R^n))}&\le C\norm{F}_{L^{\tilde p'}(\R;L^{\tilde q'}(\R^n))},\\
\label{corollary_2_3}
\bignorm{\int_0^te^{-i(t-s)H}F(s)ds}_{L^2(\R;L^{2^* ,\infty}(\R^n))}&\le C\norm{F}_{L^{\tilde p'}(\R;L^{\tilde q'}(\R^n))}
\end{align}
for all $\psi\in L^2(\R^n)$ and $F\in L^{\tilde p'}(\R;L^{\tilde q'}(\R^n))$. 
\end{corollary}

\begin{remark}
\label{remark_10}
The non-endpoint estimates \eqref{corollary_2_1} and \eqref{corollary_2_2} (with an appropriate spectral projection) hold for a more general operator of the form $H+\wtilde V$ with a radially symmetric potential $\wtilde V$ such that $|\wtilde V(x)|\le C\<x\>^{-\mu}$ with some $\mu>3$, under certain spectral conditions at zero energy for the two-dimensional operator $-\Delta_{\R^2}+\wtilde V(|x|)$ (see Section \ref{section_generalization}). 
\end{remark}

%

Theorem \ref{theorem_1} also implies the following weighted $L^2$-estimate. 

\begin{corollary}
\label{corollary_3}
There exists $C>0$ such that, for any $w\in L^{n,2}(\R^n)$ and $\psi\in L^2(\R^n)$, 
\begin{align}
\label{corollary_3_1}
\norm{we^{-itH}\psi}_{L^2(\R^{1+n})}\le C\norm{w}_{L^{n,2}(\R^n)}\norm{\psi}_{L^2(\R^n)}.
\end{align}
\end{corollary}

\begin{remark}
 \eqref{corollary_3_1} provides an alternative proof of the absolute continuity of the spectrum of $H$. 
Indeed, if $M_w$ denotes the multiplication operator by $ w$ then  \eqref{corollary_3_1} and \cite[Theorem XIII. 23]{ReSi} show that $\overline{\Ran(M_w)}\subset \H_{\mathrm{ac}}(H)$ for any $w\in L^{n,2}(\R^n)$. Taking $w=\<x\>^{-2}\in L^{n,2}(\R^n)$ for instance, we see that $L^2(\R^n)=\overline{\Ran(M_w)}$ and hence $H$ is purely absolutely continuous. 
\end{remark}

Finally we obtain the following negative result.
\begin{theorem}	
\label{theorem_4}
There exists $\psi\in L^2_{\mathrm{rad}}(\R^n)$ such that,  for any $1\le q<\infty$, $e^{-itH}\psi$ does not belong to $ L^2(\R;L^{2^* ,q}(\R^n))$. Moreover, for any $C>0$, there exists $F\in L^2(\R;L^{2_*,1}_{\mathrm{rad}}(\R^n))$  such that \begin{align}
\label{theorem_4_2}
\bignorm{\int_0^te^{-i(t-s)H}F(s)ds}_{L^2(\R;L^{2^* ,\infty}(\R^n))}>C\norm{F}_{L^{2}(\R;L^{2_*,1}(\R^n))}.
\end{align}
\end{theorem}

Since $L^{p,q_1}\subset L^{p,p}=L^p\subset L^{p,q_2}$ if $1\le q_1<p<q_2\le\infty$, this theorem particularly shows that usual endpoint Strichartz estimates can fail as in the two dimensional free case and that Theorem \ref{theorem_1} is sharp. We however stress that the counterexamples in Theorem \ref{theorem_4} are given by radial functions, while the negative results in the two-dimensional free case are given by non-radial functions (see \cite{Mon,Tao1}). It is worth noting that the dispersive estimate of the form 
$
\norm{e^{-itH_\sigma}}_{L^1\to L^\infty}\le C|t|^{-n/2}$, 
which is stronger than Strichartz estimates, for the operator $H_\sigma=-\Delta-\sigma|x|^{-2}$ fails in general as soon as $\sigma>0$ (see \cite{FFFP1,FFFP3}). 
\bigskip

\noindent{\it Notation}. Throughout the paper we use the following notation. $\<x\>$ stands for $\sqrt{1+|x|^2}$. $\<f,g\>:=\int f(x)\overline{g(x)}dx$ denotes the inner product in $L^2(\R^n)$. We write $L^p_tL^q_x=L^p(\R;L^q(\R^n))$ for short. For positive constants $A$ and $B$, $A\sim B$ means that there exists universal constants $C_1>C_2>0$ such that $C_2B\le A\le C_1B$. For $1<p<\infty$ and $1\le q\le\infty$, $L^{p,q}(\R^n)$ denotes the Lorentz space equipped with the norm $\norm{\cdot}_{L^{p,q}(\R^n)}$ satisfying 
$$
\norm{f}_{L^{p,q}(\R^n)}\sim \norm{t d_f(t)^{1/p}}_{L^q(\R_+,t^{-1}dt)},
$$
where $d_f(t):=\mu(\{x\in \R^n\ |\ |f(x)|>t\})$ is the distribution function of $f$. 
We use the convention $L^{\infty,\infty}=L^\infty$. For $1\le p,p_1,p_2<\infty$ and $1\le q,q_1,q_2\le\infty$ satisfying $1/p=1/p_1+1/p_2$, $1/q=1/q_1+1/q_2$, we have H\"older's inequality
\begin{equation}
\begin{aligned}
\label{Holder}
\norm{fg}_{L^{p,q}}\le C\norm{f}_{L^{p_1,q_1}}\norm{g}_{L^{p_2,q_2}},\quad
\norm{fg}_{L^{p,q}}\le C\norm{f}_{L^{\infty}}\norm{g}_{L^{p,q}}.
\end{aligned}
\end{equation}
It is not hard to see that $|x|^{-\alpha}\in L^{n/\alpha,\infty}(\R^n)$ for $0<\alpha\le n$. We refer to  \cite{Gra} for more details on Lorentz spaces
\bigskip

The rest of the paper is organized as follows. Section \ref{section_3} is devoted to the proof of Theorem \ref{theorem_1} and Corollaries \ref{corollary_2} and \ref{corollary_3}. The proof of Theorem \ref{theorem_4} is given in Section \ref{section_counterexample}. Section \ref{section_generalization} discusses a generalization of Corollaries \ref{corollary_2}. In appendix \ref{appendix_A}, we prove the  uniform estimate for the weighted resolvent $|x|^{-1}P^\perp (H-z)^{-1}P^\perp|x|^{-1}$.
 
\section{Proof of Theorem \ref{theorem_1}}
\label{section_3}
We first prove the weak-type endpoint estimate for radial data.
\begin{proposition}
\label{proposition_3}
For all $\psi\in L^2(\R^n)$, $e^{-itH }P\psi$ satisfies \eqref{theorem_1_1}. 
\end{proposition}

\begin{proof}
Recall that $H$ restricted to radial functions is unitarily equivalent to $-\Delta_{\R^2}$  restricted to radial functions by means of the unitary map 
\begin{align}
\label{unitary_1}
U:L^2_{\text{rad}}(\R^n)\ni f\mapsto (\omega_n/\omega_2)^{1/2}|x|^{(n-2)/2}f\in  L^2_{\text{rad}}(\R^2),
\end{align}
namely one has $$Hu=-U^*\Delta_{\R^2}Uu,\quad u\in PD(H),$$ where $U^*f=(\omega_n/\omega_2)^{-1/2}|x|^{-(n-2)/2}f$. 
This formula follows from the identity 
$$
\<-|x|^{-(n-2)/2}\Delta_{\R^2}|x|^{(n-2)/2}f,g\>=\left\langle \Big(-\frac{d^2}{dr^2}-\frac{n-1}{r}\frac{d}{dr}-\frac{(n-2)^2}{4r^2}\Big)f,g\right\rangle=Q_H(f,g)
$$ 
for all radially symmetric functions $f,g\in C_0^\infty(\R^n)$ and the fact that $C_0^\infty(\R^n)$ is dense in $D(\overline Q_H)$, where $r=|x|$.  In particular, one has 
\begin{align}
\label{proof_proposition_3_1}
e^{-itH}P\psi=U^*e^{it\Delta_{\R^2}}U P\psi.
\end{align}
By virtue of this formula, 
we can use \eqref{Tao} to obtain
\begin{align*}
\norm{e^{-itH}P\psi}_{L^2(\R;L^{2^* ,\infty}(\R^n))}
&\le C\norm{|x|^{-(n-2)/2}}_{L^{2^* ,\infty}(\R^n)}\norm{Ue^{-itH}P\psi}_{L^2(\R;L^\infty(\R^n))}\\
&\le C\norm{e^{it\Delta_{\R^2}}U P\psi}_{L^2(\R;L^\infty(\R^2))}\\
&\le C\norm{U P\psi}_{L^2(\R^2)}\\
&\le C\norm{\psi}_{L^2(\R^n)}
\end{align*}
and the assertion follows. 
\end{proof}

It remains to deal with the non-radial part for which one can consider more general potentials:

\begin{assumption}
\label{assumption_1}
$V\in L^1_{\mathrm{loc}}(\R^n)$ is a real-valued radial function satisfying the following. 
\begin{itemize}
\item[(1)] $|x|^2(x\cdot\nabla)^\ell V\in L^\infty(\R^n)$ for $\ell=0,1$. 
\item[(2)] There exists a constant $\nu>0$ such that 
$$
|x|^2V(x)\ge -(n-1)^2/4+\nu,\quad 
-|x|^2(V(x)+x\cdot \nabla V(x))\ge -(n-1)^2/4+\nu. $$
\item[(3)] $-\Delta+V$ is lower semi-bounded in the sense that, with some constant $C>0$, 
$$
\<(-\Delta+V)u,u\>\ge -C\norm{u}_{L^2}^2,\quad  u\in C_0^\infty(\R^n). 
$$
\end{itemize}
\end{assumption}
Under these conditions, 
the sesquilinear form $\<(-\Delta+V)u,v\>$ is symmetric, lower semi-bounded and closable on $C_0^\infty(\R^n)$. Let $H_V=-\Delta+V$ be a unique self-adjoint operator corresponding to the closure of this form. It is not hard to check that $V(x)=-C_{\Hardy}|x|^{-2}$ fulfills the above conditions. In particular, the following proposition particularly implies \eqref{theorem_1_2} and \eqref{theorem_1_3}.

\begin{proposition}	
\label{proposition_4}
Let $V$ satisfy Assumption \ref{assumption_1}. Then one has
\begin{align}
\label{proposition_4_1}
\norm{e^{-itH_V }P^\perp \psi}_{L^2(\R;L^{2^* ,2}(\R^n))}
&\le C\norm{\psi}_{L^2(\R^n)},\\
\label{proposition_4_2}
\bignorm{\int_0^te^{-i(t-s)H_V}P^\perp F(s)ds}_{L^2(\R;L^{2^* ,2}(\R^n))}
&\le C\norm{F}_{L^2(\R;L^{2_*,2}(\R^n))}.
\end{align}
\end{proposition}

The key ingredient in the proof of this proposition is the following weighted $L^2$ estimate:

\begin{proposition}
\label{proposition_5}
Let $V$ be as above. Then one has
\begin{align}
\label{proposition_5_00}
\norm{|V|^{1/2}e^{-itH_V}P^\perp \psi}_{L^2(\R^{1+n})}&\le C\norm{\psi}_{L^2(\R^n)},\\
\label{proposition_5_0}
\bignorm{|V|^{1/2}\int_0^t e^{-i(t-s)H_V}|V|^{1/2}P^\perp F(s)ds}_{L^2(\R^{1+n})}&\le C\norm{F}_{L^2(\R^{1+n})}. 
\end{align}
\end{proposition} 

\begin{proof}
Note that, since $V$ is radially symmetric, $P^\perp$ commutes with all of $|x|^{-1}$, $|V|^{1/2}$ and $e^{-it H_V}$. Furthermore, since $P^\perp=(P^\perp)^2$ and $|x|^2V\in L^\infty$, it suffices to show that
\begin{align}
\label{proposition_5_11}
\norm{|x|^{-1} e^{-itH_V} P^\perp \psi}_{L^2(\R^{1+n})}&\le C\norm{\psi}_{L^2(\R^n)},\\
\label{proposition_5_1}
\bignorm{|x|^{-1}P^\perp\int_0^t e^{-i(t-s)H_V}P^\perp |x|^{-1} F(s)ds}_{L^2(\R^{1+n})}&\le C\norm{F}_{L^2(\R^{1+n})}. 
\end{align}
Then the smooth perturbation method by Kato \cite{Kat} (see also \cite[Theorem 2.3]{Dan}) allows us to deduce both \eqref{proposition_5_11} and \eqref{proposition_5_1} from the uniform resolvent estimate:
\begin{align}
\label{proposition_5_2}
\sup_{z\in \C\setminus\R}\norm{|x|^{-1}P^\perp(H_V-z)^{-1}P^\perp|x|^{-1}}_{L^2(\R^n)\to L^2(\R^n)}<\infty
\end{align}
whose proof is rather technical and postponed to Appendix \ref{appendix_A}. 
\end{proof}

\begin{proof}[Proof of Proposition \ref{proposition_4}]
We follow closely the argument in \cite{BPST2} (see also \cite{BoMi}). Let us set $$
\Gamma_0 F(t)=\int_0^te^{i(t-s)\Delta}F(s)ds,\quad
\Gamma F(t)=\int_0^t e^{-i(t-s)H_V}F(s)ds.
$$ 
It was proved by \cite[Theorem 10.1]{KeTa} that $e^{it\Delta}$ and $\Gamma_0$ satisfy 
\begin{align}
\label{proof_1}
\norm{e^{it\Delta}\psi}_{L^2_tL^{2^* ,2}_x}\le C\norm{\psi}_{L^2_x},\quad
\norm{\Gamma_0F}_{L^2_tL^{2^* ,2}_x}\le C\norm{F}_{L^2_tL^{2_*,2}_x},
\end{align}
which, together with the hypothesis $|x|^2V\in L^\infty$ and H\"older's inequality, implies
\begin{equation}
\begin{aligned}
\label{proof_2}
\norm{|V|^{1/2}\Gamma_0F}_{L^2_tL^2_x}\le C\norm{F}_{L^2_tL^{2_*,2}_x},\quad
\norm{\Gamma_0|V|^{1/2} G}_{L^2_tL^{2^* ,2}_x}
\le C\norm{G}_{L^2_tL^{2}_x}.
\end{aligned}
\end{equation}
Let $a_1,a_2,b_1,b_2,c_1$ and $c_2$ be (finite) positive numbers defined by
\begin{align*}
&a_1:=\norm{e^{it\Delta}}_{L^2_x\to L^2_tL^{2^* ,2}_x},\quad
a_2:=\norm{\Gamma_0}_{L^2_tL^{2_*,2}_x\to L^2_tL^{2^* ,2}_x},\\
&b_1:=\norm{\Gamma_0|V|^{1/2}}_{L^2_tL^{2}_x\to L^2_tL^{2^* ,2}_x},\quad
b_2:=\norm{|V|^{1/2}\Gamma_0}_{L^2_tL^{2_*,2}_x\to L^2_tL^{2}_x},\\
&c_1:=\norm{|V|^{1/2}e^{-itH_V}P^\perp }_{L^2_tL^2_x\to L^2_tL^2_x},\quad
c_2:=\norm{|V|^{1/2}\Gamma  |V|^{1/2}P^\perp}_{L^2_tL^2_x\to L^2_tL^2_x}. 
\end{align*}

Now we shall show the homogeneous estimate \eqref{proposition_4_1}. By virtue of Duhamel's identity
\begin{align}
\label{Duhamel_0}
e^{-itH_V}\psi=e^{it\Delta}\psi-i\int_0^{t}e^{i(t-r)\Delta} Ve^{-irH_V}\psi dr,
\end{align}
we learn by \eqref{proof_1}, \eqref{proof_2} and \eqref{proposition_5_00} that
$$
\norm{e^{-itH_V}P^\perp \psi}_{L^2_tL^{2^* ,2}_x}\le (a_1+b_1c_1)\norm{\psi}_{L^2_x}.
$$

In order to derive the inhomogeneous estimate \eqref{proposition_4_2}, we next observe that, for any simple function $F:\R\to C_0^\infty(\R^n)$, $\Gamma$ satisfies following iterated Duhamel's identity
\begin{align}
\label{Duhamel_1}
\Gamma F(t)=\Gamma_0 F(t)-i\Gamma_0V\Gamma_0 F(t)-\Gamma_0V\Gamma V\Gamma_0 F(t)
\end{align}
Indeed, integrating usual Duhamel's identity $$e^{-i(t-s)H_V}F(s)=e^{i(t-s)\Delta}F(s)-i\int_0^{t-s}e^{i(t-s-r)\Delta} Ve^{-irH_V}F(s)dr$$ over $s\in[0,t]$ (or $s\in[t,0]$) and applying Fubini's theorem to the second term, we obtain
$$
\Gamma F(t)=\Gamma_0F(t)-i\Gamma_0V\Gamma F(t).
$$
By exchanging the roles of $-\Delta$ and $H_V$, we also have
$$
\Gamma_0 F(t)=\Gamma F(t)-i\Gamma (-V)\Gamma_0 F(t)=\Gamma F(t)+i\Gamma V\Gamma_0 F(t).
$$
These two identities imply \eqref{Duhamel_1}. Taking into account the fact that $P^\perp$ commutes with all of $V,\Gamma_0$ and $\Gamma$, we use \eqref{Duhamel_1}, \eqref{proof_1}, \eqref{proof_2} and \eqref{proposition_5_0} to obtain 
\begin{align*}
\norm{\Gamma P^\perp F}_{L^2_tL^{2^* ,2}_x}\le (a_2+b_1b_2+b_1b_2c_2)\norm{F}_{L^2_tL^{2_*,2}_x},
\end{align*}
which completes the proof since simple functions with values in $C_0^\infty(\R^n)$ are dense in $L^2_tL^{2_*,2}_x$. 
\end{proof}


Next we consider non-endpoint estimates and weighted $L^2$-estimates:

\begin{proof}[Proof of Corollary \ref{corollary_2}]
We first recall the following properties of real interpolation spaces:
\begin{itemize}
\item \cite[Theorem (I.I) in Chapter VII]{LiPe}. For $1\le p_j\le \infty$, $1\le q_j<\infty$, $1\le s_j\le \infty$, $0<\theta<1$, $1/p=(1-\theta)/p_0+\theta/p_1$ and $1/q=(1-\theta)/q_0+\theta/q_1$, one has 
$$\big(L^{p_0}_tL^{q_0,s_0}_x,L^{p_1}_tL^{q_1,s_1}_x\big)_{\theta,p}=L^p_tL^{q,p}_x.$$ 
\item \cite[Theorem 3.1.2]{BeLo}. Let $\A$ be a Banach space and $(\B_0,\B_1)$ a Banach couple and $\A$. If $T\in \mathbb B(\A,\B_0)\cap \mathbb B(\A,\B_1)$ with $\norm{T}_{\A\to \B_j}\le M_j$ then, for any $\theta\in(0,1)$ and $\sigma\in[1,\infty]$, $T$ is also bounded from $\A$ to $(\B_0,\B_1)_{\theta,\sigma}$ with $\norm{T}_{\A\to (\B_0,\B_1)_{\theta,\sigma}}\le M_0^{1-\theta}M_1^\theta$. 
\end{itemize}
Then we use \eqref{theorem_1_1} and these properties with $\A=L^2_x$, $\B_0=L^2_tL^{2^* ,\infty}_x$ and $\B_1=L^\infty_tL^2_x$ to obtain 
$$
\norm{e^{-itH}\psi}_{L^p_tL^{q,p}_x}\le C\norm{\psi}_{L^2_x}.
$$
for all non-endpoint admissible pair. By virtue of the continuous embedding $L^{q,p}\subset L^{q}$ for  $q\ge p$, this estimate implies \eqref{corollary_2_1} for $2(n+2)/n\le q<2n/(n-2)$. For $2<q<2(n+2)/n$, \eqref{corollary_2_1} follows from complex interpolation between two cases $(p,q)=(\infty,2)$ and $(p,q)=(2(n+2)/n,2(n+2)/n)$. This completes the proof of \eqref{corollary_2_1}. Inhomogeneous estimates \eqref{corollary_2_2} and \eqref{corollary_2_3} follow from a standard argument (see \cite[Section 3]{Tao1}) by using \eqref{theorem_1_1},   \eqref{corollary_2_1}, dual estimates of \eqref{corollary_2_1} and Christ-Kiselev's lemma \cite{ChKi}. 
\end{proof}

\begin{proof}[Proof of Corollary \ref{corollary_3}]
The assertion follows from \eqref{theorem_1_1} and H\"older's inequality \eqref{Holder}. 
\end{proof}

\section{Proof of Theorem \ref{theorem_4}}
\label{section_counterexample}

\subsection{The homogeneous case}
We begin with an elementary fact on Lorentz spaces:

\begin{lemma}
\label{lemma_2}
Let $n\ge2$, $1<\sigma<\infty$, $\alpha>0$  and set $w_0(r):=r^{-1}|\log r|^{-\alpha}\mathds 1_{\{r<1\}}(r)$. Then $w_0(|x|)$ belongs to $L^{n,\sigma}(\R^n)$ if and only if $\alpha>1/\sigma$. 
\end{lemma}

\begin{proof}
The distribution function $d_w(s)$ of $w$ satisfies
\begin{align*}
d_{w_0}(s)
&\sim \int_0^\infty \mathds1_{\{r^{-1}|\log r|^{-\alpha}\mathds1_{(0,1)}(r)>s\}}r^{n-1}dr\\
&=n\int_0^\infty\mathds1_{\{t^{-1/n}|\log t|^{-\alpha}n^{-\alpha}\mathds1_{(0,1)}(t)>s\}}dt.
\end{align*}
Thus $w_0(|x|)\in L^{n,\sigma}(\R^n)$ if and only if $\wtilde w_0:=r^{-1/n}|\log r|^{-\alpha}\mathds1_{(0,1)}(r)$ belongs to $L^{n,\sigma}(\R)$. Then it is known (see \cite[1.4.8 in pages 67]{Gra}) that $\wtilde w_0\in  L^{n,\sigma}(\R)$ if and only if $\sigma>1/\alpha$, {\it i.e.}, $\alpha>1/\sigma$.
\end{proof}

The following lemma provides a counterexample of space-time weighted $L^2$-estimates for $e^{it\Delta}$ in two dimensions. 
\begin{lemma}
\label{lemma_3_0}
Let $\alpha\le1/2$ and $\psi\in L^2_{\mathrm{rad}}(\R^2)$ satisfy $\norm{\widehat \psi}_{L^2(\{|x|<1\})}>0$.
Then $w_0(|x|) e^{it\Delta_{\R^2}}\psi\notin L^2(\R^{1+2})$. 
In particular, the estimate
\begin{align*}
\norm{w_0(|x|)e^{it\Delta_{\R^2}}\psi}_{L^2(\R^{1+2})}\le C\norm{\psi}_{L^2(\R^2)}
\end{align*}
does not hold true. Here $\widehat\psi$ denotes the Fourier transform of $\psi$. 
\end{lemma}

\begin{proof}
We may assume $\alpha=1/2$ without loss of generality. Set $w_\delta(r)=r^{-1}|\log r|^{-1/2}\mathds1_{[\delta,1-\delta)}(r)$ for $\delta\ge0$. In order to derive the assertion, it suffices to show that $w_\delta(|x|) e^{it\Delta_{\R^2}}\psi\in L^2(\R^{1+2})$ for each $\delta>0$ and $\norm{w_\delta(|x|) e^{it\Delta_{\R^2}}\psi}_{L^2(\R^{1+2})}\to \infty$ as $\delta\to0$. 

Taking into account the fact that, for any radial $\varphi\in L^2(\R^2)$, both $\widehat\varphi$ and $e^{it\Delta_{\R^2}}\varphi$ are also radial, 
one can write $f_\psi(|\xi|)=\widehat\psi(\xi)$ for some $f_\psi\in L^2(\R_+,rdr)$. Since $\norm{f_\psi}_{L^2([0,1],sds)}>0$ by assumption, we may assume $\norm{f_\psi}_{L^2([0,1],sds)}=1$. 
Then a direct computation yields
\begin{align*}
e^{it\Delta_{\R^2}}\psi(x)
&=\frac{1}{2\pi}\int e^{ix\cdot\xi}e^{-it|\xi|^2}f_\psi(|\xi|)d\xi\\
&=\frac12\int_0^\infty \Big(\frac{1}{2\pi}\int_0^{2\pi}e^{i\rho^{1/2}|x|\cos\theta}d\theta\Big) f_\psi(\rho^{1/2})e^{-it\rho}d\rho,
\end{align*}
where we have made the change of variables $\rho=|\xi|^2$ in the second line. We set
$$
J_0(z):=\frac{1}{2\pi}\int_0^{2\pi}e^{iz\cos\theta}d\theta,\quad
h(\rho,x):=J_0(\rho^{1/2}|x|)w_\delta(|x|)f_\psi(\rho^{1/2})
$$
for $\rho\ge0$ and $h(\rho,x)=0$ for $\rho<0$. Then Plancherel's theorem shows 
\begin{align}
\nonumber
\norm{w_\delta e^{it\Delta_{\R^2}}\psi}_{L^2_tL^2_x}^2
&=\pi\norm{h(\cdot,x)}_{L^2(\R_\rho;L^2(\R_+,rdr))}^2\\
\nonumber
&=\pi\int_0^\infty\int_0^\infty J_0(\rho^{1/2}r)^2w_\delta (r)^2|f_\psi(\rho^{1/2})|^2rdr d\rho\\
\label{proof_lemma_3_1}&\ge
2\pi\inf_{s\in[0,1]}\int_0^\infty J_0(sr)^2w_\delta(r)^2rdr.
\end{align}
Consider the integral 
\begin{align*}
K(s):=\int_0^\infty J_0(sr)^2w_\delta (r)^2rdr=\int_\delta^{1-\delta} J_0(sr)^2r^{-1}|\log r|^{-1}dr. 
\end{align*}
It is easy to see that the Bessel function  $J_0(z)$ is bounded on $[0,\infty)$. 
In particular, we have 
$$
K(s)\le C\int_\delta^{1-\delta}r^{-1}|\log r|^{-1}dr\le C(|\log(\log\delta)|+|\log(\log(1-\delta))|),\quad s\ge0,
$$
which, together with the condition $f_\psi\in L^2(\R_+,rdr)$, implies $w_\delta e^{it\Delta_{\R^2}}\psi\in L^2(\R^{1+2})$. On the other hand, if $0\le z\le 1$, then $J_0(z)^2\ge c_0$ with some $c_0>0$ and hence $K(s)$ satisfies
$$
K(s)\ge c_0\int_\delta^{1-\delta}r^{-1}|\log r|^{-1}dr=c_0\int_{\log \delta}^{\log(1-\delta)}\frac{dt}{|t|}
\ge C|\log(\log\delta)|,\quad 0\le s\le1,
$$
which, together with \eqref{proof_lemma_3_1}, shows $\norm{w_\delta e^{it\Delta_{\R^2}}\psi}_{L^2(\R^{1+2})}\to +\infty$ as $\delta\to0$.
\end{proof}


\begin{proof}[Proof of Theorem \ref{theorem_4}: The homogeneous case]
Since $L^{p,q_1}\subset L^{p,q_2}$ if $q_1\le q_2$, we may assume $2<q<\infty$. It is convenient to write $q={2\sigma}/{(\sigma-2)}$ with $2<\sigma<\infty$ so that $1/2=1/q+1/\sigma$. 
Assume for contradiction that $e^{-itH}\psi\in L^2(\R;L^{2^* ,q}(\R^n))$ for any $\psi\in L^2_{\mathrm{rad}}(\R^n)$. 

Fix $\psi_2\in L^2_{\mathrm{rad}}(\R^2)$ arbitrarily. Since $\psi := U^*\psi_2\in L^2_{\mathrm{rad}}(\R^n)$, H\"older's inequality imply
\begin{align}
\label{proof_theorem_4_0_1}
\norm{we^{-itH}\psi}_{L^2(\R^{1+n})}
\le C\norm{w}_{L^{n,\sigma}(\R^n)}\norm{e^{-itH}\psi}_{L^2(\R;L^{2^* ,q}(\R^n)}<\infty
\end{align}
for any $w\in L^{n,\sigma}(\R^n)$. On the other hand, 
assuming $w$ is also radially symmetric and using the formula \eqref{proof_proposition_3_1}, $we^{-itH}\psi$ satisfies
\begin{equation}
\begin{aligned}
\label{proof_theorem_4_0_2}
\norm{we^{-itH}\psi}_{L^2(\R^{1+n})}
=\big|\big|{U^*we^{it\Delta_{\R^2}}\psi_2}\big|\big|_{L^2(\R^{1+n})}
=\norm{we^{it\Delta_{\R^2}}\psi_2}_{L^2(\R^{1+2})}.
\end{aligned}
\end{equation}
Hence \eqref{proof_theorem_4_0_1} and \eqref{proof_theorem_4_0_2} imply $we^{it\Delta_{\R^2}}\psi_2\in L^2(\R^{1+2})$ 
for any $w\in L^{n,\sigma}_{\mathrm{rad}}(\R^n)$. 

Now if we take $1/\sigma<\alpha\le 1/2$ then $w_0$ defined in Lemma \ref{lemma_2} belongs to $L^{n,\sigma}_{\mathrm{rad}}(\R^n)$ and hence $w_0e^{it\Delta_{\R^2}}\psi_2\in L^2(\R^{1+2})$ by the above argument,  which contradicts with Lemma \ref{lemma_3_0}. 
\end{proof}


\subsection{The inhomogeneous case}

We begin by recalling two known results on the two-dimensional free Schr\"odinger equation. 

\begin{lemma}
\label{lemma_3}
There exists $C>0$ such that, for any $\psi,w\in L^2(\R^2)$, 
$$
\norm{we^{it\Delta_{\R^2}}\psi}_{L^2(\R^{1+2})}\le C\norm{w}_{L^2(\R^2)}\norm{\psi}_{L^2(\R^2)}
$$
\end{lemma}

\begin{proof}
It was shown by \cite{Pla} that 
$$
\norm{e^{it\Delta_{\R^2}}\psi}_{L^\infty(\R^2_x;L^2(\R_t))}\le C\norm{\psi}_{L^2(\R^2)}
$$
which, together with H\"older's inequality, clearly implies the desired bound. 
\end{proof}

\begin{lemma}
\label{lemma_4}
Let $\kappa=-iz^{1/2}$ with $z=-\kappa^2$ for $\Re z<0$. Then
\begin{align}
\label{log}
(-\Delta_{\R^2}+\kappa^2)^{-1}=-\frac{\log \kappa}{2\pi} P_0+O(1),\quad |\kappa|\to0
\end{align}
 in $\mathbb B(L^2(\R^2,\<x\>^{s}dx),L^2(\R^2,\<x\>^{-s}dx))$ for $s>3/2$, where $P_0f:=\int_{\R^2} f(x)dx$. 
\end{lemma}

\begin{proof}
See, {\it e.g.}, \cite[Lemma 5]{Sch}. 
\end{proof}

\begin{proof}[Proof of Theorem \ref{theorem_4}: The inhomogeneous case]
Assume for contradiction that
$$
\bignorm{\int_0^t e^{-i(t-s)H}F(s)ds}_{L^2(\R;L^{2^*,\infty}(\R^n))}\le C\norm{F}_{L^2(\R;L^{2_*,1}(\R^n))}
$$
with some $C>0$ independent of $F$. 
Taking  into account the fact that $s\in [-T,T]$ if $t\in[-T,T]$ and $s\in [0,t]$ (or $s\in [t,0]$), one may replace the time interval $\R$ in the both sides by $[-T,T]$ for any $T>0$. Moreover, using H\"older's inequality, we obtain
\begin{align}
\label{proof_theorem_4_1}
\bignorm{w\int_0^te^{-i(t-s)H}F(s)ds}_{L^2([-T,T];L^{2}(\R^n))}&\le C\norm{w^{-1}F}_{L^{2}([-T,T];L^2(\R^n))},
\end{align}
for any $w\in L^{n,2}(\R^n)$ with $w^{-1}\in L^2_{\loc}(\R^n)$ and $F\in L^2([-T,T];L^2_{\mathrm{rad}}(\R^n,w^{-2}dx))$.  
For simplicity, we take $w(x)=\<x\>^{-2}$. Plugging the formula \eqref{proof_proposition_3_1} into \eqref{proof_theorem_4_1} yields
\begin{align}
\label{proof_theorem_4_2}
\bignorm{\<x\>^{-2}\int_0^te^{i(t-s)\Delta_{\R^2} }F_2(s)ds}_{L^2([-T,T];L^{2}(\R^2))}&\le C\norm{\<x\>^2F_2}_{L^{2}([-T,T];L^2(\R^2))}
\end{align}
for all $F_2\in L^2([-T,T];L^2_{\mathrm{rad}}(\R^2,\<x\>^4dx))$. Now we claim that \eqref{proof_theorem_4_2} implies
\begin{align}
\label{proof_theorem_4_3}
\norm{\<x\>^{-2}(-\Delta_{\R^2}+\kappa^2)^{-1}g}_{L^2(\R^2)}\le C\norm{\<x\>^{2}g}_{L^2(\R^2)}
\end{align}
on $L^2_{\mathrm{rad}}(\R^2,\<x\>^4dx)$ with some $C>0$ being independent of $\kappa$. Then \eqref{proof_theorem_4_3} clearly contradicts with the logarithmic singularity as $|\kappa|\to0$ in  \eqref{log}. 

It remains to show that \eqref{proof_theorem_4_2} implies \eqref{proof_theorem_4_3}.   
Let $f\in \S(\R^2)$ be radially symmetric, $z\in \C\setminus[0,\infty)$ and set $u(t)=e^{-izt}f$. Since $u(t)$ solves
$$
i\partial_tu=-\Delta_{\R^2}u+F_2,\quad u|_{t=0}=f,
$$
where $F_2=-e^{-izt}(-\Delta_{\R^2}-z)f\in L^{2}([-T,T];L^2_{\mathrm{rad}}(\R^2,\<x\>^4dx))$, 
Lemma \ref{lemma_3} and \eqref{proof_theorem_4_2} imply
\begin{align}
\label{proof_theorem_4_4}
||\<x\>^{-2}u||_{L^2([-T,T];L^2(\R^2))}\le C ||f||_{L^2(\R^2)}+C ||\<x\>^2F||_{L^{2}([-T,T];L^2(\R^2))}, 
\end{align}
where $C$ is independent of $T$. By virtue of the specific formula of $u$ and $F$, one can compute
\begin{align*}
||\<x\>^{-2}u||_{L^2([-T,T];L^2(\R^2))}
&=\gamma(z,T)||\<x\>^{-2}f||_{L^2(\R^2)},\\
|| \<x\>^2 F ||_{L^2([-T,T];L^2(\R^2))}
&= \gamma(z,T) ||\<x\>^2(-\Delta_{\R^2}-z)f||_{L^2(\R^n)},
\end{align*}
where $\gamma(z,T):= || e^{izt} ||_{L^2([-T,T])} \geq \sqrt{T} , $
since $ |e^{izt}| \geq 1 $ either on $ [0,T] $ or $ [-T,0] $. In particular, $\gamma(z,T)\to \infty$ as $T\to  \infty$ for each $z$, so dividing by $\gamma(z,T)$ and letting $T\to  \infty$ in \eqref{proof_theorem_4_4}, we obtain
$$
\norm{\<x\>^{-2}f}_{L^2(\R^n)}\le C\norm{\<x\>^2(-\Delta_{\R^2}-z)f}_{L^2(\R^2)},\quad z\in \C\setminus[0,\infty),
$$
for all $f\in \S(\R^2)$ such that $f$ is radial. Now we plug $f=(-\Delta_{\R^2}-z)^{-1}g$ to obtain \eqref{proof_theorem_4_3} for all radial $g\in \S(\R^2)$. Finally, a density argument yields \eqref{proof_theorem_4_3} for all $g\in L^2_{\mathrm{rad}}(\R^2,\<x\>^4dx)$. 
 \end{proof}

\section{A generalization of Corollary \ref{corollary_2}}
\label{section_generalization}
Here we extend the non-endpoint estimates \eqref{corollary_2_1} and \eqref{corollary_2_2} to more general operators of the form 
$$
H_V=-\Delta+V,\quad V(x)=-C_{\Hardy}|x|^{-2}+\wtilde V(|x|),
$$
where $\wtilde V(r)$ is a real-valued $C^1$ function on $\R_+$ satisfying $|\wtilde V(r)|\le C\<r\>^{-\mu}$ for some $\mu>3$ at least. Let $H_V$ is defined as a unique self-adjoint operator associated to a lower semi-bounded closable form $\<(-\Delta+V)u,u\>$ on $C_0^\infty(\R^n)$. 

In order to state the result, we introduce several notation. Similarly to the proof of Theorem \ref{theorem_1}, the following two-dimensional Schr\"odinger operator plays a key role: 
$$
\wtilde H=-\Delta+\wtilde V(|x|),\quad D(\wtilde H)=\H^1(\R^2).
$$ 
Since $\wtilde V$ is radially symmetric, we see that $P_n$ (resp. $P_2$) commutes with $H_V$ (resp. $\wtilde H$) and that $H_V$ is unitarily equivalent to $\wtilde H$ on the space of radial functions in the sense that 
\begin{align}
\label{lemma_generalization_1_0}
H_Vu=U^*\wtilde HUu,\quad u\in P_nD(H_V),
\end{align}
where $P_n:L^2(\R^n)\to L^2_{\mathrm{rad}}(\R^n)$ is an orthogonal projection and $U:L^2_{\mathrm{rad}}(\R^n)\to L^2_{\mathrm{rad}}(\R^2)$ is a unitary map defined by \eqref{projection_1} and \eqref{unitary_1}, respectively. 

Under the above condition on $\wtilde V$, it is known that $\wtilde H$ is self-adjoint on $L^2(\R^2)$. Suppose in addition that zero energy is not an eigenvalue of $\wtilde H$. Then the spectrum of $H$ is purely absolutely continuous on $[0,\infty)$ and is pure point on $(-\infty,0)$ with finitely many eigenvalues of finite multiplicities (see \cite[Chapter XIII]{ReSi} and \cite{Sto}). Let $P_{\mathrm{pp}}(\wtilde H)$ (resp. $P_{\mathrm{ac}}(\wtilde H)$) be the projection onto the pure point (resp. absolutely continuous) subspace of $\wtilde H$. Since $P_2$ commutes with $\wtilde H$, we see that $P_2$, $P_{\mathrm{pp}}(\wtilde H)$ and $P_{\mathrm{ac}}(\wtilde H)$ mutually commute by the spectral theorem. 
Define two bounded operators $\mathbb P_{\mathrm{pp}},\mathbb P_{\mathrm{ac}}:L^2(\R^n) \to L^2_{\mathrm{rad}}(\R^n)$ by
\begin{align*}
\mathbb P_{\mathrm{pp}}:=P_nU^*P_2P_{\mathrm{pp}}(\wtilde H)P_2U P_n,\quad
\mathbb P_{\mathrm{ac}}:=P_nU^*P_2P_{\mathrm{ac}}(\wtilde H)P_2U P_n. 
\end{align*}

\begin{lemma}
\label{lemma_generalization_1}
$\mathbb P_{\mathrm{pp}}$ and $\mathbb P_{\mathrm{ac}}$ are orthogonal projections such that they commute with $H_V$ and $\Id=\mathbb P_{\mathrm{pp}}+\mathbb P_{\mathrm{ac}}+P^\perp_n$. 
Moreover, $\mathbb P_{\mathrm{pp}}$, $\mathbb P_{\mathrm{ac}}$ and $P_n^\perp$ mutually commute and $$\mathbb P_{\mathrm{pp}}\mathbb P_{\mathrm{ac}}=\mathbb P_{\mathrm{pp}}P_n^\perp=\mathbb P_{\mathrm{ac}}P_n^\perp=0.$$ 
\end{lemma}

\begin{proof}
Recall that $P_2$, $P_{\mathrm{pp}}(\wtilde H)$ and $P_{\mathrm{ac}}(\wtilde H)$ are orthogonal projections. Then it is not hard to see that $\mathbb P_{\mathrm{pp}}$ and $\mathbb P_{\mathrm{ac}}$ are bounded, symmetric and thus self-adjoint. Since $\mathrm{Ran}\,U^*P_2\subset L^2_{\mathrm{rad}}(\R^n)$, one has $UP_nU^*P_2=P_2$ which implies
\begin{align*}
\mathbb P^2_{\mathrm{pp}}
=P_nU^*P_2P_{\mathrm{pp}}(\wtilde H)P_2P_{\mathrm{pp}}(\wtilde H)P_2U P_n
=\mathbb P_{\mathrm{pp}}. 
\end{align*}
A similar argument implies  $\mathbb P_{\mathrm{ac}}^2=\mathbb P_{\mathrm{ac}}$. Thus $\mathbb P_{\mathrm{pp}}$ and  $\mathbb P_{\mathrm{ac}}$ are orthogonal projections. 

Next, it is seen from \eqref{lemma_generalization_1_0} that
\begin{equation}
\begin{aligned}
\label{lemma_generalization_1_1}
P_2UP_nH_Vu&=\wtilde HP_2UP_nu,\quad
u\in D(H_V),\\
P_nU^*P_2\wtilde H v&=H_VP_nU^*P_2v,\quad v\in D(\wtilde H). 
\end{aligned}
\end{equation}
which imply that both $\mathbb P_{\mathrm{pp}}$ and $\mathbb P_{\mathrm{ac}}$ commute with $H_V$. Moreover, since 
$$
P_n=P_n^2=P_nU^*UP_n=P_nU^*P_2^2UP_n=P_nU^*P_2\Big(P_{\mathrm{pp}}(\wtilde H)+P_{\mathrm{ac}}(\wtilde H)\Big)P_2UP_n=\mathbb P_{\mathrm{pp}}+\mathbb P_{\mathrm{ac}},
$$
we have $\Id=\mathbb P_{\mathrm{pp}}+\mathbb P_{\mathrm{ac}}+P^\perp_n$. Finally, since $$P_nP_n^\perp=P_n^\perp P_n=0,\quad P_{\mathrm{pp}}(\wtilde H)P_{\mathrm{ac}}(\wtilde H)=P_{\mathrm{ac}}(\wtilde H)P_{\mathrm{pp}}(\wtilde H)=0,$$ $\mathbb P_{\mathrm{pp}}$, $\mathbb P_{\mathrm{ac}}$ and $P_n^\perp$ mutually commute and satisfy the last identity in the statement. 
\end{proof}

Next we recall the notion of zero resonance of $\wtilde H$ (see \cite{ErGr}). 

\begin{definition}						
\label{resonance}
Zero energy is said to be a resonance of $\wtilde H$ if there is a distributional solution $f\in L^p(\R^2)\setminus L^2(\R^2)$ with some $2< p\le \infty$ to the equation $\wtilde Hf=0$. In the case of $p=\infty$ the resonance is called an $\mathrm{s}$-wave resonance, while it is called a ${\mathrm{p}}$-wave resonance if $2<p<\infty$. 
\end{definition}

Let us now state the main result in this section. 

\begin{theorem}	
\label{theorem_generalization_1}
Let $n\ge3$ and $\wtilde V\in C^1(\R_+)$ be a real-valued function such that  
$
|\wtilde V(r)|\le C\<r\>^{-\mu}
$
with some $\mu>3$. Suppose that $V(x):=-C_{\Hardy}|x|^{-2}+\wtilde V(|x|):\R^n\to \R$ satisfies Assumption \ref{assumption_1} and that one of the following conditions holds:
\begin{itemize}
\item[{\rm (H1)}] Zero energy is neither an eigenvalue nor a resonance of $\wtilde H$. 
\item[{\rm (H2)}] $\mu>4$ and $\wtilde H$ has only an s-wave resonance at zero energy. 
\end{itemize}
Then, for any ($n$-dimensional) admissible pairs $(p,q)$ and $(\tilde p,\tilde q)$ with $p,\tilde p>2$, one has
\begin{equation}
\begin{aligned}
\label{theorem_generalization_1_1}
\norm{e^{-itH_V}(\Id-\mathbb P_{\mathrm{pp}})\psi}_{L^p(\R;L^q(\R^n))}&\le C\norm{\psi}_{L^2(\R^n)},\\
\bignorm{\int_0^t e^{-i(t-s)H_V}(\Id-\mathbb P_{\mathrm{pp}}) F(s)ds}_{L^p(\R;L^q(\R^n))}&\le C\norm{F}_{L^{\tilde p'}(\R;L^{\tilde q'}(\R^n))}
\end{aligned}
\end{equation}
for all $\psi\in L^2(\R^n)$ and $F\in {L^{\tilde p'}(\R;L^{\tilde q'}(\R^n))}$. 
\end{theorem}

\begin{remark}
If in addition $\wtilde V$ is non-negative, then it is easy to see that $\wtilde H$ has no eigenvalue and thus \eqref{theorem_generalization_1_1} hold without the spectral projection $\Id-\mathbb P_{\mathrm{pp}}$. 
\end{remark}

A key ingredient in the proof of this theorem is the following dispersive estimate:

\begin{proposition}
\label{proposition_generalization_1}
Let $\wtilde V$ be a real-valued function on $\R^2$ such that $|\wtilde V(x)|\le C\<x\>^{-\mu}$ on $\R^2$ with $\mu>3$ and $\wtilde V$ satisfies either {\rm (H1)} or {\rm (H2)}. Then
\begin{align}
\label{proposition_generalization_1_1}
\norm{e^{-it\wtilde H}P_{\mathrm{ac}}(\wtilde H)}_{L^1(\R^2)\to L^\infty(\R^2)}\le C|t|^{-1},\quad |t|\neq0.
\end{align}
In particular, for all $(p,q)\in[2,\infty]^2$ satisfying $2/p+2/q=1$ and $p>2$, one has
\begin{align}
\label{proposition_generalization_1_2}
\norm{e^{-it\wtilde H}P_{\mathrm{ac}}(\wtilde H)\psi}_{L^p(\R;L^{q,2}(\R^2))}&\le C\norm{\psi}_{L^2(\R^2)},\quad \psi\in L^2(\R^2). 
\end{align}
\end{proposition}

\begin{proof}
The dispersive estimate \eqref{proposition_generalization_1_1} is due to \cite{Sch,ErGr}. Strichartz estimates \eqref{proposition_generalization_1_2} follow from the dispersive estimate \eqref{proposition_generalization_1_1} and the Keel-Tao theorem \cite[Theorem 10.1]{KeTa}. 
\end{proof}

\begin{proof}[Proof of Theorem \ref{theorem_generalization_1}]
Let us first consider the homogeneous estimates. 
If we decompose $$e^{-itH_V}(\Id-\mathbb P_{\mathrm{pp}})\psi=e^{-itH_V}\mathbb P_{\mathrm{ac}}\psi+e^{-itH_V}P_n^\perp\psi,$$ then the desired estimate for $e^{-itH_V}P^\perp \psi$ follows from Proposition \ref{proposition_4} and complex interpolation. 

Let $(p,q)$ be an $n$-dimensional non-endpoint admissible pair, {\it i.e.}, $2/p+n/q=n/2$ and $p>2$. Take $q_2\ge2$ such that $(p,q_2)$ is a two-dimensional admissible pair, {\it i.e.}, $2/p+2/q_2=1$. Using \eqref{lemma_generalization_1_0} and \eqref{lemma_generalization_1_1}, we have
$$
e^{-itH_V}\mathbb P_{\mathrm{ac}}\psi=U^*P_2e^{-it\wtilde H}P_{\mathrm{ac}}(\wtilde H)P_2U P_n\psi .
$$
Setting $v(t)=e^{-it\wtilde H}P_{\mathrm{ac}}(\wtilde H)P_2U P_n\psi$, this formula and H\"older's inequality yield
\begin{align*}
\norm{e^{-itH_V}\mathbb P_{\mathrm{ac}}\psi}_{L^p(\R;L^q(\R^n))}
&\le C\norm{|x|^{-(n-2)\left(\frac12-\frac1q\right)}v}_{L^p(\R;L^q(\R^2))}\\
&\le C\norm{|x|^{-(n-2)\left(\frac12-\frac1q\right)}}_{L^{s,\infty}(\R^2)}\norm{v}_{L^p(\R;L^{q_2,q}(\R^2))},
\end{align*}
where $s$ is given by the identity $1/q=1/s+1/q_2$. Then, since
$$
\frac2s=\frac2q-\frac{2}{q_2}=\frac n2-\frac np-1+\frac2p=(n-2)\Big(\frac12-\frac1p\Big)
$$
and $|x|^{-2/s}\in L^{s,\infty}(\R^2)$, $\norm{|x|^{-(n-2)\left(\frac12-\frac1q\right)}}_{L^{s,\infty}(\R^2)}=\norm{|x|^{-2/s}}_{L^{s,\infty}(\R^2)}$ is finite. Furthermore, taking the continuous embedding $L^{q_2,2}\subset L^{q_2,q}$ (note that $q\ge2$) into account, \eqref{proposition_generalization_1_2} shows
$$
\norm{v}_{L^p(\R;L^{q_2,q}(\R^2))}\le C\norm{P_2U P_n\psi}_{L^2(\R^2)}\le C\norm{\psi}_{L^2(\R^n)},
$$
which completes the proof of the homogeneous estimates. 

Having Lemma \ref{lemma_generalization_1} in mind, we see that the inhomogeneous estimates follow from the homogeneous estimates and Christ-Kiselev's lemma. 
\end{proof}


\appendix

\section{A uniform resolvent estimate}
\label{appendix_A}
In what follows we assume $n\ge3$ and frequently use the notation $r=|x|$ as well as $\partial_r=|x|^{-1}x\cdot\nabla$. 
This appendix is devoted to the proof of \eqref{proposition_5_2}, namely, we shall show the following.

\begin{theorem}	
\label{theorem_A}
Let $n\ge3$ and $H_V$ be as in Proposition \ref{proposition_4}. Then \eqref{proposition_5_2} holds. \end{theorem}

The proof essentially relies on the multiplier method of \cite[Section 2]{BVZ}. We follow  closely \cite[Appendix B]{BoMi}). 
Let $f\in C_0^\infty(\R^n)$ be non-radial in the sense that $f=P^\perp f$ and $\lambda+i\ep\in \C\setminus\R$ with $\lambda,\ep\in\R$. Let $u=(H_V-\lambda-i\ep)^{-1}f\in D(H_V)$ be the solution to the Helmholtz equation
\begin{align}
\label{Helmholtz}
(H_V-\lambda- i\ep)u=f,
\end{align} 
Note that $u$ also satisfies $u=P^\perp u$ since $V$ is radially symmetric. Also note that it suffices to consider the case $\ep\ge0$ only, the proof for the case $\ep<0$ being analogous. We first prepare three key lemmas:

\begin{lemma}
\label{lemma_Hardy} {\rm(1)} Improved Hardy's inequality for non-radial functions: 
$$
\frac{n^2}{4}\norm{r^{-1}P^\perp f}_{L^2(\R^n)}^2\le \norm{\nabla f}_{L^2(\R^n)}^2,\quad f\in C_0^\infty(\R^n). 
$$
{\rm (2)} Weighted Hardy's inequality:
$$
\frac{(n-1)^2}{4}\norm{r^{-1/2}f}_{L^2(\R^n)}^2\le \norm{r^{1/2}\nabla f}_{L^2(\R^n)}^2,\quad f\in C_0^\infty(\R^n). 
$$
\end{lemma}

\begin{proof}
We refer to \cite[Lemma 2.4]{EkFr} for (1) and \cite[Proposition 8.1]{MSS} for (2) in which simple proofs can be found. 
\end{proof}

\begin{remark}
\label{remark_appendix}
It is seen from Lemma \ref{lemma_Hardy} (1) and Assumption \ref{assumption_1} (2) that $\<H_VP^\perp f,f\>\sim \norm{\nabla f}_{L^2}^2$ which 
shows $P^\perp D(\overline Q_H)\subset \H^1$, where $\H^1$ is the $L^2$-Sobolev space of order $1$.  In particular, $u=(H_V-\lambda-i\ep)^{-1}P^\perp f$ belongs to $\H^1$. This observation is useful to justify the computations in the proof of the next lemma. 
\end{remark}

\begin{lemma}
\label{lemma_appendix_C_1}
The following five identities hold: 
\begin{align}
\label{proof_C_2}
\int\Big(|\nabla u|^2+V|u|^2-\lambda |u|^2\Big)dx
&=\Re\int f\overline{u}dx,\\
\label{proof_C_3}
-\ep\int|u|^2dx
&=\Im \int f\overline{u}dx,\\
\label{proof_C_4}
\int \Big(r|\nabla u|^2-\lambda r|u|^2+rV|u|^2+\Re(\overline u\partial_ru)\Big)dx
&=\Re \int rf\overline udx,\\
\label{proof_C_5}
\int\Big(-\ep r|u|^2+\Im(\overline u\partial_ru)\Big)dx
&=\Im \int rf\overline udx,\\
\label{proof_C_6}
\int\Big(2|\nabla u|^2-(r\partial_rV)|u|^2-2\ep \Im(\overline ur\partial_ru)\Big)dx
&=\Re \int f(2r\partial_r \overline u+n\overline u)dx. 
\end{align}
In particular, we have $r^{1/2}\nabla u\in L^2$, $r^{1/2}u\in L^2$ and $r^{1/2}u\in \H^1$. 
\end{lemma}

\begin{proof} We give an outline of the proof only and refer to \cite[Appendix B]{BoMi} for more details. Note that the condition $r^2(r\partial_r)^\ell V\in L^\infty$ with $\ell=0,1$ is enough to justify following computations. \eqref{proof_C_2} and \eqref{proof_C_3} are verified by multiplying \eqref{Helmholtz} by $\overline u$, integrating over $\R^n$ and taking  the real and imaginary parts. \eqref{proof_C_4} and \eqref{proof_C_5} follow from multiplying \eqref{Helmholtz} by $r\overline u$, integrating over $\R^n$ and taking  the real and imaginary parts. 
In order to derive \eqref{proof_C_6}, we  multiply \eqref{Helmholtz} by $iA\overline u$ with $iA=x\cdot\nabla+\nabla\cdot x$, integrate over $\R^n$ and take the real part. Then one has
\begin{align*}
2\Re\<H_Vu,iAu\>
&=\<[H_V,iA]u,u\>
=\<(-2\Delta-r\partial_r V)u,u\>
=\int\Big(2|\nabla u|^2-(r\partial_rV)|u|^2\Big)dx,\\
\Re\<u,iAu\>
&=0,\\
2\ep\Re(-i\<u,iAu\>)
&=2\ep\Im \<u,iAu\>
=2\ep\Im \int u(r\partial_r \overline u)dx
=-2\ep\Im \int\overline u(r\partial_ru)dx,\\
2\Re\<f,iAu\>&=\Re\int f(2r\partial_r \overline u+n\overline u)dx,
 \end{align*}
and hence \eqref{proof_C_6} follows. Finally, by \eqref{proof_C_5} and \eqref{proof_C_4}, we have $r^{1/2}\nabla u,r^{1/2}u\in L^2$ which, together with Lemma \ref{lemma_Hardy} (2), imply $r^{1/2}u\in \H^1$. 
\end{proof}

\begin{lemma}
\label{lemma_appendix_C_2}
Let $0<\ep<\lambda$ and $v_\lambda=e^{-i\lambda^{\frac12}r}u$. Then one has
\begin{equation}
\begin{aligned}
\label{proof_C_7}
&\int \Big(|\nabla v_\lambda|^2-\partial_r(rV)|v_\lambda|^2+\ep\lambda^{-\frac12}r|\nabla v_\lambda|^2+\ep\lambda^{-\frac12}rV|v_\lambda|^2\Big)dx\\
&=\Re\int\Big(-\ep\lambda^{-\frac12}\overline ue^{i\lambda^\frac12 r}\partial_rv_\lambda+(n-1)f\overline u+\ep\lambda^{-\frac12}rf\overline u+2rf\overline{e^{i\lambda^{\frac12}r}\partial_rv_\lambda}\Big)dx. 
\end{aligned}
\end{equation}
\end{lemma}

\begin{proof}
This formula is derived by computing
$\eqref{proof_C_6}-\eqref{proof_C_2}-2\lambda^{\frac12}\times\eqref{proof_C_5}+\ep\lambda^{-\frac12}\times\eqref{proof_C_4}$ (see \cite[Appendix B]{BoMi} for more details). 
\end{proof}

\begin{proof}[Proof of Theorem \ref{theorem_A}]
It is sufficient to show that
\begin{align}
\label{proof_C_10_0}
\norm{r^{-1}u}_{L^2}\le C\norm{rf}_{L^2}
\end{align}
holds uniformly in $\lambda\in\R$ and $\ep>0$. When $\ep\ge\lambda$, \eqref{proof_C_2} and \eqref{proof_C_3} imply
\begin{align}
\label{proof_C_10_1}
\int\Big(|\nabla u|^2+V|u|^2\Big)dx\le (1+\lambda_+/\ep)\int |fu|dx\le \delta_1\norm{r^{-1}u}_{L^2}^2+\delta_1^{-1}\norm{rf}_{L^2}^2
\end{align}
for any $\delta_1>0$, where $\lambda_+=\max\{0,\lambda\}$. Taking the fact $u=P^\perp u\in \H^1$ into account, we learn by Assumption \ref{assumption_1} (2) and Lemma \ref{lemma_Hardy} (1) that
$$
\int\Big(|\nabla u|^2+V|u|^2\Big)dx\ge \int\Big(|\nabla u|^2+(-n^2/4+\ep)r^{-2}|u|^2\Big)dx\ge \nu\norm{r^{-1}u}_{L^2}^2. 
$$
Taking $\delta_1<\nu$ we obtain \eqref{proof_C_10_0}.  

Next, let $\ep<\lambda$. By Assumption \ref{assumption_1} (2) and Lemma \ref{lemma_Hardy},  the left hand side of \eqref{proof_C_7} satisfies 
\begin{align*}
&\int \Big(|\nabla v_\lambda|^2-\partial_r(rV)|v_\lambda|^2+\ep\lambda^{-\frac12}r|\nabla v_\lambda|^2+\ep\lambda^{-\frac12}rV|v_\lambda|^2\Big)dx\\
&\ge 
\nu\norm{r^{-1}u}_{L^2}^2+\nu\ep\lambda^{-\frac12}\norm{r^{-1/2}u}_{L^2}^2. 
\end{align*}
Hence it suffices to show that there exist $\delta<\nu$ and $C_\delta>0$ being independent of $\ep$ and $\lambda$ such that the right hand side of \eqref{proof_C_7} is bounded from above by 
$\delta\norm{r^{-1}u}_{L^2}^2+C\norm{rf}_{L^2}^2$. 

For the first term of the right hand side of \eqref{proof_C_7}, the Cauchy-Schwarz inequality and Hardy's inequality \eqref{Hardy} yield that there exists $C_1>0$ independent of $\ep$ and $\lambda$ such that
\begin{align*}
\ep\lambda^{-\frac12}\Big|\Re \int e^{i\lambda^{\frac12} r}(\partial_rv_\lambda) \overline udx\Big|
&\le \sqrt \ep \norm{\nabla v_\lambda}_{L^2}\norm{u}_{L^2}\\
&\le \delta_1\norm{r^{-1}u}_{L^2}^2+C_1\delta_1^{-1}\ep\norm{u}_{L^2}^2. 
\end{align*}
for any $\delta_1>0$. Here \eqref{proof_C_3} and Hardy's inequality imply
\begin{align*}
\ep\norm{u}_{L^2}^2
\le \int |fu|dx
\le \delta_1^2\norm{r^{-1}u}_{L^2}^2 + C_1'\delta_1^{-2}\norm{rf}_{L^2}^2.
\end{align*}
with some universal constant $C_1'>0$. Hence we obtain
\begin{align}
\label{proof_C_12_1}
\ep\lambda^{-\frac12}\Big|\Re \int e^{i\lambda^{\frac12} r}(\partial_rv_\lambda) \overline udx\Big|
\le (1+C_1)\delta_1\norm{r^{-1}u}_{L^2}^2+C_1C_1'\delta_1^{-3}\norm{rf}_{L^2}^2. 
\end{align}
For other terms, similar computations yield
\begin{align}
\label{proof_C_15}
\Big|(n-1)\Re\int f\overline udx\Big|
&\le \delta_1\norm{r^{-1}u}_{L^2}^2+C_2\delta_1^{-1}\norm{rf}_{L^2}^2,\\
\Big|2\Re \int rf\overline{e^{i\lambda^{\frac12}r}\partial_rv_\lambda}dx\Big|
&\le \delta_1\norm{r^{-1}u}_{L^2}^2+C_2\delta_1^{-1}\norm{rf}_{L^2}^2,\\
\label{proof_C_16}
\ep\lambda^{-\frac12}\Big|\int rf\overline udx\Big|
\le \sqrt\ep\norm{rf}_{L^2}\norm{u}_{L^2}^2
&\le \delta_1\norm{r^{-1}u}_{L^2}^2+C_3\delta_1^{-3}\norm{rf}_{L^2}^2. 
\end{align}
\eqref{proof_C_12_1}--\eqref{proof_C_16} show that the right hand side of \eqref{proof_C_7} is bounded from above by $\delta\norm{r^{-1}u}_{L^2}^2+C\norm{rf}_{L^2}^2$ with $\delta=(4+C_1)\delta_1$ and $C=(C_1C_1'+C_2+C_3+C_4)\delta_1^{-3}$. Choosing $\delta_1>0$ so small  that $(4+C_1)\delta_1<\nu$ we obtain \eqref{proof_C_10_0}, which completes the proof.
\end{proof}
\noindent
{\bf Acknowledgments.} The author would like to thank Jean-Marc Bouclet for valuable discussions and for hospitality at the Institut de Math\'ermatiques de Toulouse, Universit\'e Paul Sabatier, where this work has been done. The author also thank the anonymous referees for useful suggestions that helped to improve the presentation of this paper. He is partially supported by JSPS Grant-in-Aid for Young Scientists (B) (No. 25800083) and by Osaka University Research Abroad Program (No. 150S007).


\end{document}